\documentclass[14pt ]{amsart}
\usepackage{amsmath, amstext, amsbsy, amssymb}
\usepackage{graphicx}
\usepackage{subfigure}
\usepackage{picinpar}
\usepackage{psfrag}
\usepackage[all]{xy}

\hoffset \voffset \oddsidemargin=55pt \evensidemargin=55pt
\numberwithin{equation}{section}
\def\beq{\begin{eqnarray}}
\def\eeq{\end{eqnarray}}
\def\beqs{\begin{eqnarray*}}
\def\eeqs{\end{eqnarray*}}
\def\mN{{\mathbb N}}
\def\mz{{\mathbb Z}}
\def\mc{{\mathbb C}}

\def\dim{{\hbox{\rm dim}}}

\newfont{\df}{eufm10}

\def\mc{{\mathbb C}}

\def\sg{\frak{g}}
\def\sh{\frak{h}}
\def\ca{\mathcal A}
\def\cb{\mathcal B}
\def\mf{\mathbb F}
\title[]
{On the center of the quantized enveloping algebra of a simple Lie algebra}

\author[L. Li]{Libin Li}
\address{L. Li: College of Mathematical Science, Yangzhou University, Yangzhou, 225009, China}
\email{lbli@yzu.edu.cn}
\author[L. Xia]{Li-meng   Xia$^{\dag}$}
\address{L.Xia: Institute of Applied System Analysis, Jiangsu University, Zhenjiang, 212013, China}
\email{xialimeng@ujs.edu.cn}

\author[Y. Zhang]{Yinhuo Zhang}
\address{Y. Zhang: Department of Mathematics and Statistics, University of Hasselt, Universitaire Campus, 3590 Diepeenbeek, Belgium}
\email{yinhuo.zhang@uhasselt.be}

\thanks{$^\dag$The corresponding author.}
\thanks{This work was supported by National Natural Science Foundation of China (Grants No. 11471282, 11271131).}
\date{}

\date{}
\begin{document}
\maketitle

\begin{abstract}{ Let $\sg$ be a  finite dimensional simple complex Lie algebra and $U=U_q(\sg)$ the quantized enveloping algebra (in the sense of Jantzen) with $q$ being generic. In this paper, we show that the center $Z(U_q(\sg))$ of the quantum group $U_q(\sg)$ is isomorphic to a monoid algebra, and  that $Z(U_q(\sg))$ is a polynomial algebra if and only if $\sg$ is of  type $A_1, B_n, C_n, D_{2k+2}, E_7, E_8, F_4$  or $G_2.$  Moreover,  in case $\sg$ is of type $D_{n}$ with $n$ odd,  then $Z(U_q(\sg))$ is isomorphic to a quotient algebra of a polynomial algebra in $n+1$ variables  with one relation;  in case $\sg$ is of type $E_6$,  then $Z(U_q(\sg))$ is isomorphic to a quotient algebra of a polynomial algebra in fourteen variables with eight relations;  in case $\sg$ is of type $A_{n}$, then $Z(U_q(\sg))$ is isomorphic to a quotient algebra of a polynomial algebra described by $n$-sequences.}
\end{abstract}

\newtheorem{theo}{Theorem}[section]
\newtheorem{theorem}[theo]{Theorem}
\newtheorem{defi}[theo]{Definition}
\newtheorem{lemma}[theo]{Lemma}
\newtheorem{coro}[theo]{Corollary}
\newtheorem{prop}[theo]{Proposition}
\newtheorem{remark}[theo]{Remark}
\newtheorem{exam}[theo]{Example}

\section{Introduction}

Let $\sg$ be a  finite-dimensional simple complex Lie algebra with Cartan matrix $A=(a_{i,j})_{n\times n}$.  It is well known that the center of the universal enveloping algebra $U(\sg)$ is isomorphic to the invariant subalgebra  $S(\sh)^W$ of the Weyl group $W$  acting on  the symmetric algebra $S(\sh)$ of a Cartan subalgebra $\sh$ of  $\sg$.  The isomorphism is given by  the Harish-Chandra homomorphism. Consequently, the center of $U(\sg)$ is a polynomial algebra for any type of $\sg.$

In the quantized case,  if $q$ is generic, the structure of the center of the quantized enveloping algebra  $U_q(\sg)$ in the sense of  Drinfel$^\prime$d and Jimbo was described respectively by Drinfel$^\prime$d in \cite{Dr} and by Reshetikhin in \cite{R}. It was showed there that  the Harish-Chandra homomorphism, from the representation ring of  the tensor category of finite dimensional  modules of type one  over $U_q(\sg)$ to the center, is an isomorphism,  and can be constructed  from the canonical quasi-triangular structure of $U_q(\sg)$ (see also \cite{B} ). Like the non-quantized  case, the center of $U_q(\sg)$ is  a polynomial algebra for any type $\sg.$  Moreover,  the Drinfel$^\prime$d-Reshetikhin's construction can be generalized to any  symmetrizable Kac-Moody algebra $\sg$ as well as to any quantum affine algebra, see \cite{E}.

The situation turns to be more complicated when one considers the quantized enveloping algebra  $U_q(\sg)$ in Jantzen's sense with $q$ being generic. The Harish-Chandra homomorphism can be generalized to Jantzen's quantum group $U_q(\sg)$, and defines a homomorphism from the center $Z(U_q(\sg))$ to the Laurent polynomial algebra $U^0$  generated by the group-like elements in $U_q(\sg)$. Following the fact that the intersection of the annihilators of all finite-dimensional modules is trivial, one can  see that  the Harish-Chandra homomorphism is injective, and that the image of the  homomorphism is contained in ${(U^0)}^W$.  Now one would expect the equality here like in the classical case. However,  this  is not the case.  The image turns out to be $(U_{ev}^0)^W$,  the linear span of all $K_{\mu},$  where $\mu$ is  an even weight (see Section 2 for details). Yet, to the best of our knowledge, the minimal generating set for $(U_{ev}^0)^W$ has not been determined in general. However, if $\sg$ is of type $A_1$, then the center $Z(U_q(\sg))$ is generated by the quantum Casimir element, and $(U_{ev}^0)^W$ is isomorphic to the polynomial algebra $\mathbb{C}[X]$, where the quantum Casimir element $X$ stands for $EF+\frac {Kq^{-1}+K^{-1}q}{(q-q^{-1})^2},$ see \cite{J} for details. When $\sg$ is of types $E_8, F_4$ or $G_2$, the root lattice $Q$ coincides with the weight lattice $\Lambda$ of $\sg$. In this case, $Z(U_q(\sg))$ was proved to be a polynomial algebra (see \cite{Dr}, \cite{R}). This  leads us to the following question: for which type of a Lie algebra $\sg$,  is $(U_{ev}^0)^W$  a polynomial algebra? In general,  the invariant subalgebra  $(U_{ev}^0)^W$ is not necessarily  a polynomial algebra. In \cite{LWP, WWL}, we have proved that the center of $U_q(sl_{n})$  for $n=3, 4$ is isomorphic to a quotient algebra of polynomial algebra in $n$ variables and with one relation. Using representation theory, for any $\lambda \in Q\cap2\Lambda$, one may construct an element $z_{\lambda}\in Z(U_q(\sg))$, and obtain a basis for $Z(U_q(\sg)),$ see \cite{J}. It follows from Noether's theorem that $(U_{ev}^0)^W$ is a finitely generated algebra, see \cite{C,F, L} for details.

In this paper, we study the center of the quantized enveloping algebra  $U_q(\sg)$ in Jantzen's sense for any finite type of $\sg$. We describe explicitly the minimal generating set of $(U_{ev}^0)^W$ and classify the types of $\sg$  such that $(U_{ev}^0)^W$ is a polynomial algebra with variables $n$ (see Theorem 3.7 and Theorem 3.8, respectively), where $n$ is the rank of $\sg$. Explicitly, the center $Z(U_q(\sg))$ is isomorphic to a polynomial algebra if and only if $\sg$ is of type $A_1, B_n, C_n, D_{2k+2}, E_7, E_8, F_4$ or $G_2.$
Then we exhibit the detailed algebraic structure of the center $Z(U_q(\sg))$ for type  $D_n$ with $n$ odd (Theorem 4.1) and for  type $E_6$ (Theorem 4.2). Finally, in case $\sg$ is of type $A_n$ with $n\geq2$, we solve this problem by using the $n$-sequences.  The corresponding main results are presented in Theorem 4.4 and Theorem 4.7.

\section{Quantized enveloping algebra and even weight lattice}

Throughout the paper,  $\sg$ is a finite-dimensional simple complex Lie algebra with Cartan matrix $A=(a_{ij})_{n\times n}$. Suppose that $\sh\subset\sg$ is a Cartan subalgebra and $\Pi=\{\alpha_i|1\leq i\leq n\}\subset\sh^*$ is a prime root system. Let $(,)$ be a non-degenerate invariant symmetric bilinear form on $\sg$. The restriction to $\sh$  is non-degenerate as well.  Thus,  there exists an induced non-degenerate symmetric bilinear form over $\sh^*$,  also denoted  $(,)$.  Moreover, the following holds:
\beqs a_{ij}&=&\frac{2(\alpha_i,\alpha_j)}{(\alpha_i,\alpha_i)},\quad \forall 1\leq i,j\leq n.\eeqs

We may assume that the Dynkin diagram of $\sg$ determined  by $A$ is one of the following:

\centerline{\begin{tabular}{ll}
$A_n:$&\xymatrix{\circ_1\ar@{-}[r]&\circ_2\ar@{-}[r]&\circ\cdots\cdots\circ_{n-1}\ar@{-}[r]&\circ_n}\\
$B_n:$&\xymatrix{\circ_1\ar@{-}[r]&\circ_2\ar@{-}[r]&\circ\cdots\cdots\circ\ar@{-}[r]&\circ_{n-1}\ar@{=>}[r]&\circ_n}\\
$C_n:$&\xymatrix{\circ_1\ar@{-}[r]&\circ_2\ar@{-}[r]&\circ\cdots\cdots\circ\ar@{-}[r]&\circ_{n-1}\ar@{<=}[r]&\circ_n}\\
$D_n:$&\xymatrix{\circ_1\ar@{-}[r]&\circ_2\ar@{-}[r]&\circ\cdots\cdots\circ\ar@{-}[r]&\;\circ_{n-2}\ar@{-}[r]\ar@{-}[d]&\circ_n\\
&&&\;\circ_{n-1}}
\end{tabular}}
\centerline{\begin{tabular}{ll}
$E_{6,7,8}:$&\xymatrix{\circ_1\ar@{-}[r]&\circ_3\ar@{-}[r]&\circ_4\ar@{-}[d]\ar@{-}[r]&\circ_{5}\ar@{-}[r]&\circ_6\cdots\cdots\circ\\
&&\circ_{2}}\\
$F_4:$&\xymatrix{\circ_1\ar@{-}[r]&\circ_2\ar@{=>}[r]&\circ_3\ar@{-}[r]&\circ_4}\\
$G_2:$&\xymatrix{\circ_1\equiv\!\equiv\!\equiv\!\equiv\!\rangle\circ_2}
\end{tabular}}

For $1\leq i\leq n,$ denote by $\lambda_i$ the fundamental weights corresponding to $\alpha_i^\vee\in\sh$, the coroots identified as $\frac{2\alpha_i}{(\alpha_i,\alpha_i)}$ by the linear isomorphism
\beqs \sh^*\rightarrow (\sh^*)^*=\sh,&&\alpha\mapsto(\alpha,\cdot).\eeqs

Let $Q$ and $\Lambda$ be the root lattice and the weight lattice respectively. Explicitly,
\beqs Q=\bigoplus_{i=1}^n\mz\alpha_i,&& \Lambda=\bigoplus_{i=1}^n\mz\lambda_i.\eeqs
For the correlations between the prime roots and the fundamental weights, see the  tables  given in the Appendix.


Now let $q$ be a variable. Recall that the quantized enveloping algebra $U_q=U_q(\sg)$ in Jantzen's sense (see \cite{J}) is defined as the associative $\mc(q)$-algebra  generated by $4n$ generators
$E_i,F_i, K_i, K_i^{-1} (1\leq i\leq n)$
subject to the following relations:
\beqs &&K_iK_j=K_jK_i,\,K_iK_i^{-1}=K_i^{-1}K_i=1,\\ &&K_iE_jK_i^{-1}=q^{a_{i,j}}E_j,\, K_iF_jK_i^{-1}=q^{-a_{i,j}}F_j,\\
&&[E_i,F_j]=\delta_{i,j}\frac{K_i-K_i^{-1}}{q-q^{-1}},\\
&&\sum_{s=0}^{1-a_{i,j}}\left[1-a_{i,j}\atop s\right]_{q_i}E_i^{1-a_{i,j}-s}E_jE_i^s=0,\\ &&\sum_{s=0}^{1-a_{i,j}}\left[1-a_{i,j}\atop s\right]_{q_i}F_i^{1-a_{i,j}-s}F_jF_i^s=0, {\rm if\;}i\not=j,
\eeqs
where $q_i=q^\frac{(\alpha_i,\alpha_i)}{2}$ for all $1\leq i\leq n$. 

Let $U^0$ be the subalgebra of $U_q$ generated by $K_i^{\pm1}(1\leq i\leq n)$. Obviously, $U^0$ is the Laurent polynomial algebra  over $\mc(q)$. Moreover, $U^0$ is also the group algebra of $Q$ with the canonical  basis  $\{K^\alpha\}_{\alpha\in Q}$, where $K^{\pm\alpha_i}=K_i^{\pm1}$ and $K^{\alpha}= \prod_{i=1}^nK^{a_i}$ if $\alpha=\sum_{i=1}^{n}a_i\alpha_i.$

The Weyl group $W$ is generated by the reflections:
\beqs S_{\alpha_i}:\sh^*\rightarrow\sh^*,&&\alpha\mapsto\alpha-\frac{2(\alpha_i,\alpha)}{(\alpha_i,\alpha_i)}\alpha_i.\eeqs
It is well known that the action of $W$ maps $Q$ (also $\Lambda$) to itself.
So $W$ acts on $U^0$ in a natural way:
\beqs\omega\cdot K^\alpha=K^{\omega(\alpha)},\quad\forall\alpha\in Q.\eeqs

Set
\beqs U_{ev}^0&=&\bigoplus_{\alpha\in Q\cap2\Lambda}\mc(q) K^\alpha.\eeqs
Then $U_{ev}^0$ is a subalgebra of $U^0$ and is stable under the action of $W$.  Denote by $(U_{ev}^0)^W$ the $W$-invariants in $U_{ev}^0$.  By the Harish-Chandra Theorem (see Section 6.25 and  the beginning of Chapter 6 in \cite{J}), we have the following
\begin{theo}
The algebra $Z(U_q)$ is isomorphic to $(U_{ev}^0)^W$.
\end{theo}

For convenience, we set:
\beqs \alpha^\diamond&=&\left\{\begin{array}{ll}\displaystyle\sum_{i\;{\rm is\;odd}}\alpha_i,&{\rm\;if\;}\sg{\rm\;is\;of\;type\;}A_n,\\
\alpha_{n-1}+\alpha_{n},&{\rm\;if\;}\sg{\rm\;is\;of\;type\;}D_n,\end{array}\right.\eeqs

\begin{prop}
\beqs Q\cap2\Lambda&=&\left\{\begin{array}{ll}2Q \ (\not=2\Lambda),&{\rm\;if\;}\sg{\rm\;is\;of\;type\;}A_{2k}{\;\rm or\;}E_6,\\
2\Lambda\  (\not=2Q), &{\rm\;if\;}\sg{\rm\;is\;of\;type\;}A_1, B_n, C_n,D_{2k+2}{\;\rm or\;}E_7,\\
2\Lambda\  (=2Q) ,&{\rm\;if\;}\sg{\rm\;is\;of\;type\;}E_8, F_4 {\;\rm or\;}G_2,\\
2Q+\mz\alpha^\diamond\ (\not=2\Lambda),&{\rm\;if\;}\sg{\rm\;is\;of\;type\;}A_{2k+1}{\;\rm or\;}D_{2k+3}\, (k\geq1).\\
\end{array}\right.\eeqs
In particular, if $\sg$ is of  type $D_{n}$ with $n$ odd, then
\beqs 2Q+\mz\alpha^\diamond=4\mz\lambda_{n-1}+4\mz\lambda_n+\mz\alpha^\diamond+\displaystyle\sum_{i=1}^{n-2}2\mz\lambda_i.\eeqs
\end{prop}
\begin{proof}
 From Table 1 in Appendix, we know that $\Lambda=Q$ if and only if $\sg$ is one of $E_8, F_4, G_2$, and $2Q\not=2\Lambda\subset Q$ if and only if $\sg$ is one of
$A_1, B_n, C_n, D_{2k+2}, E_7$. It follows that  $Q\cap2\Lambda=2\Lambda$ in these cases.

Let $\alpha=\sum_{i=1}^{n}a_i\alpha_i\in Q\cap2\Lambda$. By Table 2 in Appendix, we have the following discussion.

(1) If $\sg$ is of type $A_{n}$, then
\beqs\alpha&=&(2a_1-a_2)\lambda_1+(2a_2-a_1-a_3)\lambda_2+\cdots+(2a_n-a_{n-1})\lambda_n.\eeqs
Moreover, $\alpha\in2\Lambda$ implies that \beqs &&a_2\in2\mz,\; a_1+a_3\in 2\mz,\; a_4\in2\mz,\cdots,\\
&&a_{n-1}\in2\mz,\; a_n+a_{n-2}\in 2\mz,\; 2a_{n-3}\in2\mz,\cdots.\eeqs
When $n=2k$, all $a_i$ are even. Hence, $Q\cap2\Lambda\subset 2Q$ and we have that $Q\cap2\Lambda=2Q$ because $2Q\subset Q$ and $2Q\subset 2\Lambda.$
When $n=2k+1$, all $a_{2i}$ are even and $a_1-a_{2i+1}\in2\mz$. Thus we have $Q\cap2\Lambda=2Q+\mz\alpha^\diamond$.

(2) If $\sg$ is of type $D_{n}$ with $n$ odd, then we have
\beqs\alpha&=&(2a_1-a_2)\lambda_1+(2a_2-a_1-a_3)\lambda_2+\cdots+(2a_{n-3}-a_{n-4}-a_{n-2})\lambda_{n-3}\\
&&+(2a_{n-2}-a_{n-3}-a_{n-1}-a_n)\lambda_{n-2}+(2a_{n-1}-a_{n-2})\lambda_{n-1}+(2a_n-a_{n-2})\lambda_n.\eeqs
and
\beqs &&a_1,a_2,a_3,\cdots,a_{n-2}, a_{n-1}+a_n\in2\mz.\eeqs
Since $n$ is odd,  we have $Q\cap2\Lambda\subset2Q+\mz\alpha^\diamond$. Now $\alpha^\diamond=2(\lambda_{n}+\lambda_{n-1}-\lambda_{n-2})\in Q\cap2\Lambda$ and $2Q\subset Q\cap2\Lambda$. It follows that
$Q\cap2\Lambda=2Q+\mz\alpha^\diamond$.

Furthermore, since $(2a_{n-1}-a_{n-2})+(2a_n-a_{n-2})=(2a_{n-1}+2a_n)-2a_{n-2}\in4\mz$ and $2\lambda_i\in Q$ ($1\leq i\leq n-2$),  we deduce that
\beqs 2Q+\mz\alpha^\diamond=4\mz\lambda_{n-1}+4\mz\lambda_n+\mz\alpha^\diamond+\displaystyle\sum_{i=1}^{n-2}2\mz\lambda_i.\eeqs

(3) If $\sg$ is of type $E_{6}$, we have
\beqs \alpha&=&(2a_1-a_3)\lambda_1+(2a_2-a_4)\lambda_2+(2a_3-a_1-a_4)\lambda_3+(2a_4-a_2-a_3-a_5)\lambda_4\\
&&+(2a_5-a_4-a_6)\lambda_5+(2a_6-a_5)\lambda_6\in 2\Lambda.\eeqs
However, $2a_1-a_3, 2a_2-a_4, 2a_6-a_5\in2\mz$ imply  $a_3, a_4, a_5\in2\mz$, and $2a_3-a_1-a_4, 2a_5-a_4-a_6, 2a_4-a_2-a_3-a_5\in2\mz$ imply  $a_1, a_2, a_6\in2\mz$.
Thus,  $Q\cap2\Lambda\subset 2Q$,   and hence $Q\cap2\Lambda=2Q$.
\end{proof}

\section{Minimal generators of $Z(U_q(\sg))$}

Let $\Lambda^+$ be the set of dominant integral weights and set $\Psi_{ev}=Q\cap2\Lambda^+$. 
 For any $\alpha\in Q\cap2\Lambda$, we define
\beqs f^\alpha&=&\frac1{|W|}\sum_{\omega\in W}K^{\omega(\alpha)}.\eeqs

\begin{prop}
The invariant subalgebra $(U_{ev}^0)^W$ is spanned by $f^\alpha (\alpha\in\Psi_{ev})$.
\end{prop}
\begin{proof}
Clearly, $f^\alpha\in(U_{ev}^0)^W$ for all $\alpha\in \Psi_{ev}$.

For any $\sum_{\mu}c_\mu K^\mu\in(U_{ev}^0)^W$ and $\omega\in W$, we have
\beqs \sum_{\mu}c_\mu K^\mu=\omega\cdot\Big(\sum_{\mu}c_\mu K^\mu\Big)=\sum_{\mu}c_\mu K^{\omega(\mu)}.\eeqs
It follows that
\beqs \sum_{\mu}c_\mu K^\mu&=&\frac{1}{|W|}\sum_{\omega\in W}\omega\cdot\Big(\sum_{\mu}c_\mu K^\mu\Big)\\
&=&\frac{1}{|W|}\Big(\sum_{\mu}c_\mu\sum_{\omega\in W} K^{\omega(\mu)}\Big)=\frac{1}{|W|}\sum_{\mu}c_\mu f^\mu.\eeqs
\end{proof}


For a dominant integral weight $\lambda$, we  denote by $L(\lambda)$  the irreducible highest weight $\sg$-module with highest weight $\lambda$. Recall that the character of $L(\lambda)$ is given by  (for example, see \cite[Subsection 22.5]{Hum}):
\beq\label{char} \chi(L(\lambda))&=&\sum_{\mu\in P(\lambda)}\dim L(\lambda)_\mu e^\mu,\eeq
where $L(\lambda)_\mu$ is the weight space with respect to weight $\mu$ and $P(\lambda)=\{\mu|\dim L(\lambda)\not=0\}$ is the weight set of $L(\lambda)$.

Let $R(\sg)=\mc(q)\otimes_\mz r(\sg)$, where  $r(\sg)$ is the Green ring of $\sg$. It is well known that $r(\sg)$ is a $\mz$-algebra with a basis:  $\{\chi(L(\lambda))|\lambda\in\Lambda^+\}$.  The multiplication of $r(\sg)$ is given by
\beq\label{char-prod} \chi(L(\lambda))\cdot\chi(L(\mu))&=&\sum_{\gamma}c^\gamma_{\lambda,\mu}\chi(L(\gamma)),\eeq
where $c^\gamma_{\lambda,\mu}$ is the multiplicity of $L(\gamma)$ appeared in the decomposition of the tensor product $L(\lambda)\otimes L(\mu)$.

The following lemma is the Chevalley-Shephard-Todd theorem, see \cite{C} and \cite{ST}.

\begin{lemma}
The  Green ring $r(\sg)$ $($and therefore $R(\sg) )$ is a polynomial algebra in variables $\chi(L(\lambda_i)), 1\leq i\leq n$.
\end{lemma}
For convenience,  we set $\Psi=\{\lambda/2|\lambda\in\Psi_{ev}\}.$  Then we have the following lemma.

\begin{lemma}
There exists an algebra monomorphism $\theta: (U^0_{ev})^W\rightarrow R(\sg)$ defined by
\beqs \theta\left(\sum_{\omega\in W}K^{\omega(\alpha)}\right)&=&\sum_{\omega\in W}e^{\omega(\alpha/2)},\quad \forall \alpha\in Q\cap2\Lambda.\eeqs
Moreover, $Im(\theta)={\rm span}\{\chi(L(\lambda))|\lambda\in\Psi\}$.\end{lemma}
\begin{proof}
First, $\chi(L(\lambda))$ is invariant under the $W$-action (see \cite[Theorem 21.2]{Hum}).

For each weight $\lambda\in\Lambda$, let $O_\lambda$ denote the $W$-orbit containing $\lambda$, and let $|O_\lambda|$ denote the length of $O_\lambda$.  Then we have:
\beqs \frac{|W|}{|O_\lambda|}\sum_{\mu\in O_\lambda} e^\mu&=&\sum_{\omega\in W}e^{\omega(\lambda)}.\eeqs

Since for any $\lambda\in \Lambda$ there exists a $w_0\in W$ such that $\lambda'=w_{0}(\lambda)\in \Lambda^+$ and $\lambda=0$ if and only if $\lambda'=0,$ we have
$$\sum_{\omega\in W}e^{\omega(\lambda)}= \sum_{\omega\in W}e^{\omega(\lambda')}.
 $$
Thus, in order to show that $\sum_{\omega\in W}e^{\omega(\lambda)}\in R(\sg)$, it suffices to prove the equality for $\lambda\in\Lambda^+$. Next we prove it by induction.

Clearly, if $\lambda\in\Lambda^+$ and $\lambda-\alpha_i\not\in\Lambda^+$ for each $i$, then we have:
\beqs \sum_{\omega\in W}e^{\omega(\lambda)}&=&\frac{|W|}{|O_\lambda|}\chi({L(\lambda)})\in R(\sg).\eeqs

Note that by Eq (\ref{char}), if $\lambda\in\Lambda^+$, then
\beqs \sum_{\omega\in W}e^{\omega(\lambda)}&=&\frac{|W|}{|O_\lambda|}\left(\chi({L(\lambda)})-\sum_{\mu\in P(\lambda)\setminus O_\lambda}\dim L(\lambda)_\mu e^{\mu}\right)\\
&=&\frac{|W|}{|O_\lambda|}\chi (L({\lambda}))-\sum_{0\leq \mu<\lambda}\frac{|O_\mu|}{|O_\lambda|}\dim L(\lambda)_\mu\sum_{\omega\in W}e^{\omega(\mu)},\eeqs
where $\mu<\lambda$ means $\lambda-\mu=\sum_{i=1}^n c_i\alpha_i\not=0$ for some nonnegative integers $c_i$.

If $\lambda\in\frac{Q}2\cap\Lambda$ and $\dim L(\lambda)_\mu>0$, then $\mu\in\lambda+Q\subset \frac{Q}2\cap\Lambda$ since $Q\subset\frac{Q}2\cap\Lambda$.

Under the assumption of induction that $\sum_{\omega\in W}e^{\omega(\mu)}\in R(\sg)$ for all $0\leq\mu<\lambda$, we have $\sum_{\omega\in W}e^{\omega(\lambda)}\in R(\sg)$.

Now we have showed that the linear map $\theta$ is well defined.
Moreover, the linear map  $\theta$ preserves the multiplication  because $K^\alpha K^{\alpha'}=K^{\alpha+\alpha'}$ and $e^\mu e^{\mu'}=e^{\mu+\mu'}$. So $\theta$ is also an algebra homomorphism.
Finally, the injectivity of $\theta$ is  easy to check.
\end{proof}

Now define the set $\Psi_{\rm min}=\{\lambda\in\Psi\setminus\{0\}\,|\, \lambda\not=\mu_1+\mu_2,\forall \mu_1, \mu_2\in\Psi\setminus\{0\}\}.$

\begin{lemma}
$\Psi_{\rm min}$ is finite.
\end{lemma}
\begin{proof}
By observing Table 1, for all $1\leq i\leq n$, we find  that a minimal positive integer $c_i$ exists such that $c_i\lambda_i\in \Psi$. If $\lambda=\sum_{i=1}^nc'_i\lambda_i\in\Psi_{\rm min}$ such that
$c'_{i_0}\geq c_{i_0}$ for some $i_0$, then $\lambda-c_{i_0}\lambda_{i_0}\in\Psi$ and
\beqs \chi(L(\lambda-c_{i_0}\lambda_{i_0}))\cdot\chi(L(c_{i_0}\lambda_{i_0}))&=&\chi(L(\lambda))+\sum_{0\leq \gamma<\lambda}c_{\lambda-c_{i_0}\lambda_{i_0},c_{i_0}\lambda_{i_0}}^\gamma\chi(L(\gamma)).\eeqs
In fact, $0\leq \gamma<\lambda$ implies $0\leq(\gamma,\gamma)<(\lambda,\lambda)$ and only finitely many $\gamma\in\Psi$ such that $0\leq \gamma<\lambda$.
By induction on the square length $(\lambda,\lambda)$ of $\lambda$, we obtain that $\sum_{i=1}^nc'_i\lambda_i\in\Psi_{\min}$ implies that $c_i'<c_i$ for all $i$.  Hence, $\Psi_{\min}$ is finite.
\end{proof}

Because $\Psi_{\rm min}$ is finite and  the partial order  is transitive, we may write $\Psi_{\rm min}=\{\mu_1,\cdots, \mu_m\}$ such that $i<j$ for all $\mu_i>\mu_j$, where $m=|\Psi_{\rm min}|$.

\begin{lemma}  The following hold:
\begin{enumerate}
\item[(i)] $m=n$ if and only if $\sg$ is of types $A_1, B_n, C_n, D_{n=2k+2}, E_7, E_8, F_4, G_2$.
\item[(ii)] $m=14$ if $\sg$ is of type $E_6$.
\item[(iii)] $m=n+1$ if $\sg$ is of type $D_{n=2k+3}$.
\item[(iv)] $m\geq n+1$ if $\sg$ is of type $A_n$ with $n\geq 2$.
\end{enumerate}
\end{lemma}
\begin{proof}
If $\sg$ is of type $A_n$ with $n\geq 2$, it holds that
\beqs\lambda_{i}=\frac{n+1-i}{n+1}(\alpha_1+2\alpha_2+\cdots+(i-1)\alpha_{i-1})+\frac{i}{n+1}(\alpha_n+2\alpha_{n-1}+\cdots+(n-i)\alpha_{i+1})+\frac{i(n+1-i)}{n+1}\alpha_i,\eeqs
for all $i$. Then $\frac{n+1}{(n+1,2i)}$ is the minimal positive integer such that $\frac{n+1}{(n+1,2i)}\lambda_i\in Q/2$ and we have the following inclusion:
\beqs\Big\{\lambda_1+\lambda_n, \frac{n+1}{(n+1,2i)}\lambda_{i}\,\Big|\,1\leq i\leq n \Big\}\subset\Psi_{\rm min},\eeqs
which implies $m\geq n+1>n$, and hence (iv) holds.

Except for the type $A_n(n\geq2)$,  we can list the elements of $\Psi_{\rm min}$ explicitly as follows:
\beqs \Psi_{\rm min}&=&\left\{\begin{array}{ll}\{\lambda_1,\cdots,\lambda_n\},\quad {\rm if\;}\sg\;{\rm is\;of\;type}\;A_1,B_n,C_n,D_{n=2k+2}, E_7, E_8, F_4, G_2;\\
\{3\lambda_1,\lambda_2, 3\lambda_3, \lambda_4, 3\lambda_5, 3\lambda_6,\lambda_{1}+\lambda_{3}, \lambda_{1}+\lambda_{6}, \lambda_{3}+\lambda_{5},\\
\quad\lambda_{5}+\lambda_{6},\lambda_1+2\lambda_5,2\lambda_1+\lambda_5,\lambda_3+2\lambda_6,2\lambda_3+\lambda_6\},\quad {\rm if\;}\sg\;{\rm is\;of\;type}\; E_{6};\\
\{\lambda_1,\cdots,\lambda_{n-2},2\lambda_{n-1},2\lambda_n,\lambda_{n-1}+\lambda_n\},\quad {\rm if\;}\sg\;{\rm is\;of\;type}\; D_{n=2k+3}.
\end{array}\right.\eeqs
So this lemma holds.
\end{proof}


In the sequel,  we set  $z_i=\chi(L(\lambda_i))$  for $1\leq i\leq n$.

\begin{lemma}
The set $\{\chi(L(\lambda)) | \lambda\in\Psi_{\rm min}\}$ is a minimal generating set of $Im(\theta)$.
\end{lemma}
\begin{proof} 
For each $\mu_i=\sum_{j=1}^nb_{ij}\lambda_j\in\Psi_{\rm min}$,  we let $y_i=\chi(L(\mu_i))$ and $x_i=\prod_{j=1}^nz_j^{b_{ij}}$.

By Lemma 3.3, Eq. (3.2) and the definition of $\Psi_{\rm min}$, the subalgebra $Im(\theta)$ is  generated by the elements $\{y_1,\cdots, y_m\}$. 

If $\lambda=\sum_{i=1}^na_i\lambda_i\in \Psi$, then by Eq. (3.2)  the following holds:
\beq\label{induction} \prod_{i=1}^nz_i^{a_i}&=&\chi(L(\lambda))+\sum_{\gamma<\lambda}d^\gamma_{\lambda}(a_1,\cdots,a_n)\chi(L(\gamma)),\eeq
where $d^\gamma_{\lambda}(a_1,\cdots,a_n)$ is the multiplicity of $L(\gamma)$ appeared in the decomposition of the tensor product
\beqs L(\lambda_1)^{\otimes a_1}\otimes\cdots\otimes L(\lambda_n)^{\otimes a_n}.\eeqs
In particular, if $d^\gamma_{\lambda}(a_1,\cdots,a_n)\not=0$, then $\lambda-\gamma\in Q$ and $\gamma\in\Psi$. Therefore $\chi(L(\gamma))$  belongs to $Im(\theta)$ by Lemma 3.3. It follows that $\prod_{i=1}^nz_i^{a_i}\in Im(\theta)$.


By the definition of $x_i$, the algebra $A$ generated by $x_1,\cdots, x_m$ is isomorphic to the monoid algebra, where the monoid is generated by $\mu_1,\cdots,\mu_m$ in $\Psi$.
By the definition of $\Psi_{\rm min}$, 
the monomial $x_i$ can not be algebraically represented by other monomials $x_j(j\not=i)$. It follows that $\{x_1,\cdots, x_m\}$ is a minimal generating set of $A$.

Furthermore, Eq. (\ref{induction}) implies the following:
\beqs\left.\begin{array}{lcl} x_1&\in& y_1+\mz[y_2,\cdots,y_m],\\
x_2&\in& y_2+\mz[y_3,\cdots,y_m],\\
\cdots&&\cdots\\
x_{m-1}&\in& y_{m-1}+\mz[y_m],\\
x_m&\in& y_m+\mz,\end{array}\right\}&\Rightarrow&\left\{\begin{array}{lcl} y_1&\in& x_1+\mz[x_2,\cdots,x_m],\\
y_2&\in& x_2+\mz[x_3,\cdots,x_m],\\
\cdots&&\cdots\\
y_{m-1}&\in& x_{m-1}+\mz[x_m],\\
y_m&\in& x_m+\mz.\end{array}\right.\eeqs
So the $\mc(q)$-algebra generated by $y_1,\cdots,y_m$ is the algebra $A$, and $y_1,\cdots, y_m$ is also a minimal generating set of $A$.  This completes the proof.
\end{proof}

\begin{theo}
 The center $Z(U_q(\sg))$ is isomorphic to $\mc(q)[\Psi]$, where $\mc(q)[\Psi]$ is the algebra with a basis indexed by $\Psi$ and its multiplication is determined by the operation of $\Psi.$  \end{theo}
 \begin{proof}
 Following the proof of Proposition 3.6, we know that the center $Z(U_q)$ is isomorphic to the $\mc(q)$-algebra $A$ generated by the elements $x_1,\cdots, x_m$. In particular, $\{x_1,\cdots, x_m\}\subset r(\sg)$ is  a  minimal generating set of $Im(\theta)$. 

For all nonnegative integers $b_1,\cdots, b_m$, the product $x_1^{b_1}\cdots x_m^{b_m}$ is a monomial in variables $z_1,\cdots, z_n$.  We define  $z_\lambda=\prod_{i=1}^nz_i^{a_i}$, for each $\lambda=\sum_{i=1}^na_i\lambda_i\in\Lambda^+$.

By definition, every monomial $z_\lambda\in Im(\theta)$ must be a monomial in variable $x_i$. Thus,  $z_\lambda\in Im(\theta)$  if and only if $z_\lambda=x_1^{b_1}\cdots x_m^{b_m}$ for some nonnegative integers $b_1,\cdots, b_m$. This implies that $$\sum_{i=1}^na_i\lambda_i=\sum_{i=1}^mb_j\mu_j\in\frac{Q}2,$$ or equivalently, $\sum_{i=1}^n2a_i\lambda_i\in Q$. It follows that  $Im(\theta)$ has a basis:
$${\frak {B}}=\left\{z_\lambda \Big| \lambda\in \Psi \right\}.$$


Now $0\in \Psi$ and $z_\lambda z_\mu=z_{\lambda+\mu}$ for all $\lambda,\mu\in \Psi$.  So ${\frak{B}}$ is a multiplicative monoid and   the map
$${\frak{B}}\rightarrow \Psi, z_\lambda\mapsto \lambda$$
is  an isomorphism of monoids. It follows that $ Z(U_q)\cong Im(\theta)=\mc(q)[\frak{B}]\cong\mc(q)[\Psi]$,  and hence this theorem holds.
\end{proof}

Now we can determine for which type of Lie algebra $\sg$ the center $Z(U_q(\sg))$ is a polynomial algebra.

\begin{theo} The center $Z(U_q(\sg))$ is isomorphic to a polynomial algebra if and only if $\sg$ is of type $A_1, B_n, C_n, D_{2k+2}, E_7, E_8, F_4$ or $G_2.$ \end{theo}
\begin{proof}
We also use the symbol $x_i=\prod_{j=1}^nz_j^{b_{ij}}$ for each $\mu_i=\sum_{j=1}^nb_{ij}\lambda_j\in\Psi_{\rm min}$. By Theorem 3.7, the center $Z(U_q(\sg))$ is isomorphic to the algebra $\mc(q)[\Psi]$. However, the algebra $\mc(q)[\Psi]$ (and  $\mc(q)[\frak{B}]$ ) is a polynomial algebra if and only if the generators $x_i$ are algebraically $\mc(q)$-independent, that is, no polynomial $f(t_1,\cdots,t_m)\in\mc(q)[t_1,\cdots,t_m]\setminus\{0\}$ exists such that $f(x_1,\cdots,x_m)=0$.  By the definition of $x_i$, it follows that $Z(U_q(\sg))$ is a polynomia algebra  if and only if $|\Psi_{\min}|=n$.  By Lemma 3.4, we obtain that $Z(U_q)$ is a polynomial algebra if and only if  $\sg$ is one of the  types  $A_1, B_n, C_n, D_n, E_7, E_8, F_4$  and  $G_2$, where $n$ is even.
So the theorem holds. Moreover, in these polynomial cases, we have  $Im(\theta)=R(\sg)$.
\end{proof}


 \section{Other cases}

In this section, we shall consider the case when the center $Z(U_q(\sg))$ is not isomorphic to a polynomial algebra. For $\sg$ being of type $D_{n},$  we have

\begin{theo}
If $\sg$ is of type $D_{n}$ with $n$ odd, then $Z(U_q(\sg))$ is isomorphic to the quotient ring $$\mc(q)[t_1,t_2,\cdots,t_{n+1}]/(t_{n-1}t_{n}-t_{n+1}^2).$$
\end{theo}
\begin{proof} 

By Lemma 3.5,  $\Psi_{\rm min}=\{\lambda_1,\cdots,\lambda_{n-2},2\lambda_{n-1},2\lambda_n,\lambda_{n-1}+\lambda_n\}$.
Moreover,  we have \beqs \lambda_{n-1}+\lambda_n&=&\alpha_1+2\alpha_2+\cdots+(n-2)\alpha_{n-2}+\frac{n-1}{2}(\alpha_{n-1}+\alpha_{n})\in Q,\\
2\lambda_{n-1}&=&\alpha_1+2\alpha_2+\cdots+(n-2)\alpha_{n-2}+\frac{n}{2}\alpha_{n-1}+\frac{n-2}{2}\alpha_{n})\not\in Q,\\
2\lambda_{n-1}&=&\alpha_1+2\alpha_2+\cdots+(n-2)\alpha_{n-2}+\frac{n-2}{2}\alpha_{n-1}+\frac{n}{2}\alpha_{n})\not\in Q.\eeqs
By Tabel 1, if $\lambda$ is a dominant weight and $\lambda<\lambda_{n-1}+\lambda_n$, then $\lambda=\lambda_i$ for some $i\leq n-2$ and $i$ is even.
Similarly, if $\lambda$ is a dominant weight and $\lambda<2\lambda_{n-1}$ or $\lambda<2\lambda_n$, then $\lambda=\lambda_i$ for some $i\leq n-2$ and $i$ is odd.
So \beqs\chi(L(\lambda_{n-1}+\lambda_n))&=&z_{n-1}z_n+c_{\lambda_{n-1},\lambda_{n}}^{0}+\sum_{i\leq n-2 {\rm \;is\;even}}c_{\lambda_{n-1},\lambda_n}^{\lambda_i}z_i,\\
\chi(L(2\lambda_{n-1}))&=&z_{n-1}^2+\sum_{i\leq n-2 {\rm \;is\;odd}}c_{\lambda_{n-1},\lambda_{n-1}}^{\lambda_i}z_i,\\
\chi(L(2\lambda_n))&=&z_n^2+\sum_{i\leq n-2 {\rm \;is\;odd}}c_{\lambda_{n},\lambda_n}^{\lambda_i}z_i. \eeqs
These yield that the space spanned by $\{1, \chi(\mu_i)|i=1,\cdots,n+1\}$ coincides with the space spanned by $\{1, z_1,\cdots,z_{n-2}, z_{n-1}^2, z_n^2, z_{n-1}z_n\}$.
Thus,  the map \beqs t_{n-1}\mapsto z_{n-1}^2,\; t_n\mapsto z_n^2,\; t_{n+1}\mapsto z_{n-1}z_n,\; t_i\mapsto z_i, \; 1\leq i\leq n-2,
\eeqs
defines an algebra epimorphism.

Now it suffices to prove the following isomorphism
\beqs \mc(q)[t_{n-1}, t_n, t_{n+1}]/(t_{n-1}t_{n}-t_{n+1}^2)\cong \mc(q)[z_{n-1}^2, z_n^2, z_{n-1}z_n].\eeqs
It is a well-known result.
\end{proof}

\begin{theo}
If $\sg$ is of type $E_6$, then the center $Z(U_q(\sg))$ is isomorphic to the quotient ring ${\frak R}={\mathcal S}/I$, where
${\mathcal S}$ is the polynomial algebra $\mc(q)[t_1,t_2,t_3,t_4,t_5,t_6,t_7,t_8,t_9,t_{10},t_{11},t_{12},t_{13},t_{14}]$ and $I$ is the ${\mathcal S}$-ideal generated by eleements
$$t_1t_3-t_7^3,t_1t_6-t_8^3,t_3t_5-t_9^3,t_8t_9-t_7t_{10}, t_7t_9^2-t_3t_{11}, t_7^2t_9-t_3t_{12}, t_7t_8^2-t_1t_{13}, t_7^2t_8-t_1t_{14}.$$
\end{theo}
\begin{proof}
By Lemma 3.5 and the proof of Lemma 3.6, the algebra $Im(\theta)$ is generated by  the following fourteen elements
\beqs z_1^3, z_2, z_3^3, z_4, z_5^3, z_6^3, z_1z_3, z_1z_6, z_3z_5, z_5z_6, z_1z_5^2, z_1^2z_5, z_3z_6^2, z_3^2z_6.\eeqs
Thus, the map $\phi$ given by
\beqs &&\phi(t_1)=z_1^3, \phi(t_2)=z_2, \phi(t_3)=z_3^3, \phi(t_4)=z_4, \phi(t_5)=z_5^3, \phi(t_6)=z_6^3,\\ &&\phi(t_7)=z_1z_3, \phi(t_8)=z_1z_6, \phi(t_9)=z_3z_5, \phi(t_{10})=z_5z_6,\\
&&\phi(t_{11})=z_1z_5^2, \phi(t_{12})=z_1^2z_5, \phi(t_{13})=z_3z_6^2, \phi(t_{14})=z_3^2z_6\eeqs
defines an algebra epimorphism from ${\mathcal S}$ to $Im(\theta)$ and $I\subseteq Ker\phi$.

 Let $S_1=\mc(q)[t_1,\cdots,t_6]$ and $S_2=\mc(q)[z_1^3, z_2, z_3^3, z_4, z_5^3, z_6^3]$. The map $\phi|_{S_1}$ is an algebra isomorphism from $S_1$ to $S_2$ since $z_1,\cdots, z_6$ are algebraically independent. So $S_1\cap Ker\phi=\{0\}$.

Now let $T_i=S_i\setminus\{0\}$ for $i=1,2$.  It is obvious that $\phi$ can be extended to an algebra epimorphism (denoted  $\overline{\phi}$) from $T_1^{-1}{\mathcal S}$ to $T_2^{-1}Im(\theta)$ as follows:
\beqs \phi(a^{-1}b)=\phi(a)^{-1}\phi(b),\;\forall a\in T_1, b\in{\mathcal S}.\eeqs
In particular, $J:=T_1^{-1}I\subseteq Ker\overline{\phi}$.

By observation, $J\cap{\mathcal S}=I$. 
If $T_1^{-1}{\mathcal S}/J$ is a field, then $Ker\overline{\phi}=J$, and hence
\beqs Ker\phi\subseteq Ker\overline{\phi}\cap{\mathcal S}=J\cap{\mathcal S}=I,\eeqs
which implies $Ker\phi=I$ and ${\frak R}\cong Im(\theta)\cong Z(U_q(\sg))$.




Next all elements are considered in $T_1^{-1}{\mathcal S}/J$, and we prove that the quotient ring is a field.

Let $F_0=\mc(q)(t_1,\cdots,t_6)=T_1^{-1}S_1$, the fraction field of $S_1$. Then $F_1:=F_0[t_7]$ is isomorphic to $F_0[t]/(t^3-t_1t_3)$, a field extension of $F_0$, since $t^3-t_1t_3$ is irreducible in $F_0[t]$. Specially, we may write $F_1=F_0[(t_1t_3)^\frac13]$.

Similarly, we have 
\beqs &&F_2:=F_1[t_8]=F_0[(t_1t_3)^\frac13, t_8]=F_0[(t_1t_3)^\frac13, (t_1t_6)^\frac13],\\
&&F_3:=F_2[t_9]=F_0[(t_1t_3)^\frac13, (t_1t_6)^\frac13,t_9]=F_0[(t_1t_3)^\frac13, (t_1t_6)^\frac13, (t_3t_5)^\frac13],\eeqs
both of $F_2$ and $F_3$ are fields since $t^3-t_1t_6$ is irreducible in $F_1[t]$ and $t^3-t_3t_5$ is irreducible in $F_2[t]$.  Moreover,
\beqs t_{10}=t_7^{-1}t_8t_9,\; t_{11}=t_3^{-1}t_7t_9^2,\; t_{12}=t_3^{-1}t_7^2t_9,\; t_{13}=t_1^{-1}t_7t_8^2,\; t_{14}=t_1^{-1}t_7^2t_8.\eeqs
It follows that $T_1^{-1}{\mathcal S}/J=F_3$ is a field.
The proof is completed.
\end{proof}

In the following, we assume that $\sg$  is the simple Lie algebra of type $A_n$ with $n\geq 2$. In this case, $\alpha_i=-\lambda_{i-1}+2\lambda_i-\lambda_{i+1}$, where  $\lambda_0=\lambda_{n+1}=0$.
We also use the notation $r:=\frac{n+1}{\gcd(n+1,2)}$.

As free abelian groups, there exists an isomorphism $\xi:\Lambda\rightarrow\mz^{n}$ defined by
\beqs \sum_{i=1}^nt_i\lambda_i\mapsto (t_1,\cdots, t_n),\eeqs
where the sum of $n$-tuples ${\bf t}$ and ${\bf s}$ is defined by ${\bf t}+{\bf s}=(t_1+s_1,\cdots, t_n+s_n)$.  For every $n$-tuple ${\bf t}=(t_1,t_2,\cdots,t_n)\in\mz^n$, we define $|{\bf t}|=t_1+2t_2+\cdots+nt_n$.


\begin{lemma}
Let $\lambda\in\Lambda$. Then $\lambda\in Q/2$ if and only if $|\xi(\lambda)|\in r\mz$.
\end{lemma}
\begin{proof}

For any  $\lambda=\sum_{i=1}^nt_i\lambda_i$, we have
\beqs \lambda&=&t_1\lambda_1-\sum_{i=2}^nt_i(\alpha_{i-1}+2\alpha_{i-2}+\cdots+(i-1)\alpha_{1}-i\lambda_1)\\
&=&(\sum_{i=1}^nit_i)\lambda_1-\big(t_n\alpha_{n-1}+(2t_n+t_{n-1})\alpha_{n-2}+\cdots+((n-1)t_n+\cdots+t_{2})\alpha_1\big).\eeqs
Notice that  $r=\frac{n+1}{\gcd(n+1,2)}$ and $\lambda_1=\frac{n}{n+1}\alpha_1+\frac{1}{n+1}(\alpha_n+2\alpha_{n-1}+\cdots+(n-1)\alpha_{2}).$
It is easy to see that $\lambda\in Q/2$ if and only if $(\sum_{i=1}^nit_i)\lambda_1\in Q/2$ if and only if $r|\sum_{i=1}^nit_i$. 
Thus, the lemma holds.\end{proof}

An $n$-sequence is an $n$-tuple ${\bf t}=(t_1,\cdots,t_n)$ with $t_i\in\mN$.
Then the restriction $\xi_+:=\xi|_{\Lambda^+}$ is an isomorphism between semigroups $\Lambda^+$ and the set of $n$-sequences of type $r$. 

Take $k\in\mz_+$. An $n$-sequence ${\bf t}$ is said to be of type $k$ if $|{\bf t}|$ is a multiple of $k$.  Let ${\bf 0}$ denote the zero $n$-sequence $(0,\cdots,0)$.  An $n$-sequence ${\bf t}$ of type $k$ is called minimal if ${\bf t}\not={\bf 0}$ and it is  not a sum of two non-zero $k$-type $n$-sequences.

\begin{theo} Let $\sg$ be of type $A_n$ with $n\geq 2$.  Then the  cardinality of $\Psi_{\rm min}$ equals the number of minimal $n$-sequences of type  $r$, where $\Psi_{\rm min}=\{\lambda\in\Psi\setminus\{0\}\,|\, \lambda\not=\mu_1+\mu_2,\forall \mu_1, \mu_2\in\Psi\setminus\{0\}\}$. 
\end{theo}
\begin{proof}
We claim that $\xi$ defines a bijection between $\Psi_{\min}$ and the set of minimal $n$-sequences of type $r$. In fact, if $\lambda\in\Psi_{\min}$, then $\xi(\lambda)$ is minimal. Otherwise, we may assume that $\xi(\lambda)={\bf t}'+{\bf t}''$ such that both ${\bf t}'$ and ${\bf t}''$ are of type $r$ and non-zero. Then
$\xi^{-1}({\bf t}')$ and $\xi^{-1}({\bf t}'')$ are non-zero dominant weights, and $\lambda=\xi^{-1}({\bf t}')+\xi^{-1}({\bf t}'')\not\in\Psi_{\min}$, a contradiction.

Conversely, let ${\bf t}$ be of type $r$. We show that  $\xi^{-1}({\bf t})\in\Psi_{\min}$ if ${\bf t}$ is minimal.  Suppose that $\xi^{-1}({\bf t})=\mu'+\mu''$ for some dominant weights $\mu', \mu''\in\Psi\setminus\{0\}$. Then ${\bf t}=\xi(\mu')+\xi(\mu'')$ is not minimal, a contradiction!
Since $\Psi_{\min}$ is finite, the theorem holds.
\end{proof}

\begin{lemma}
For all $k\not=1,r$, we have $d_k\lambda_1+\lambda_k\in\Psi_{\min}$, where
$d_k=2r-k$ if $r<k<n+1$ and $d_k=r-k$ for other cases.
\end{lemma}
\begin{proof}
By Lemma 4.3, we have that $d_k\lambda_1+\lambda_k\in\Psi$ implies that $r|(d_k+k)$. Let $d_k\lambda_1+\lambda_k\in\Psi_{\min}$. Then $d_k$ is the minimal positive integer such that $d_k\lambda_1+\lambda_k\in\Psi$.

If $k<r$, then $d_k=r-k$, because $d_k+k=r$. If $r<k$, then $r<k<n+1$ and $d_k=2r-k$ since $r<d_k+k=2r$.
Thus, the lemma holds.\end{proof}

Let  ${\bf e}(k)=\xi(d_k\lambda_1+\lambda_k)$ for all $k\not=1,r$.  These $n$-sequences of type $r$ are called {\bf special}. An $n$-sequence ${\bf t}$ is called {\bf single} if only one $t_i\not=0$ for some $1\leq i\leq n$. Let $r_k=\frac{r}{(r,k)}$. Then by the proof of Lemma 3.5(iv) and Theorem 4.4, $\{\xi(r_k\lambda_k)| 1\leq k\leq n\}$ is the set of  all single minimal $n$-sequences of type $r$.

Let $\mathcal{S}_n^r$ denote the set of all $n$-sequences of type $r$  and set
\beqs T=\{{\bf t}\in \mathcal{S}_n^r \mid  {\bf t}\hbox{\, is\, minimal,\, neither\, single\, and\, nor\, special}\}.\eeqs
If $r<n+1$, then $r=\frac{n+1}{2}$ and
$$\lambda_r=\frac{1}{2}(\alpha_1+2\alpha_2+\cdots+(r-1)\alpha_{r-1}+r\alpha_r+(r+1)\alpha_{r+1}+\cdots+2\alpha_{n-1}+\alpha_n)\in Q/2.$$
So $\lambda_r\in\Psi_{\min}$ and every  other $\lambda\in\Psi_{\min}$ has the form $\lambda=\sum_{k\not=r}t_k\lambda_k$.

\begin{lemma}
For every ${\bf t}=(t_1,\cdots,t_n)\in T$, it holds that
\beqs {\bf t}+\|{\bf t}\|\xi(r\lambda_1)=\sum_{k\not=1,r}t_k{\bf e}(k),
\eeqs
where $\|{\bf t}\|=\frac1{r}(-t_1+\sum_{k\not=1,r}t_kd_k)$ is a positive integer.\end{lemma}
\begin{proof}
Note  that  \beqs\sum_{k\not=1,r}t_k{\bf e}(k)={\bf t}+\xi(-t_1\lambda_1+\sum_{k\not=1,r}t_kd_k\lambda_1).\eeqs We have $-t_1\lambda_1+\sum_{k\not=1,r}t_kd_k\lambda_1\in\frac{Q}2\cap\Lambda$ and hence $r|(-t_1+\sum_{k\not=1,r}t_kd_k)$.

Since ${\bf t}\in T$, we may assume that $k_0\not=1,r$ such that $t_{k_0}\geq 1$. If $-t_1+\sum_{k\not=1,r}t_kd_k\leq0$, then
\beqs \xi^{-1}({\bf t})=d_{k_0}\lambda_1+\lambda_{k_0}+(t_1-d_{k_0})\lambda_1+(t_{k_0}-1)\lambda_{k_0}+\sum_{k\not=1,r,k_0}t_k\lambda_k.\eeqs
Because ${\bf t}$ is minimal, we have that $\xi^{-1}({\bf t})\in\Psi_{\min}$. Moreover, the facts $d_{k_0}\lambda_1+\lambda_{k_0}\in\Lambda^+\setminus\{0\}$ and $t_1-d_{k_0}\geq t_1-\sum_{k\not=1,r}t_kd_k\geq0$ yield that
\beqs (t_1-d_{k_0})\lambda_1+(t_{k_0}-1)\lambda_{k_0}+\sum_{k\not=1,r,k_0}t_k\lambda_k=0. \eeqs
Hence, ${\bf t}=\xi(d_{k_0}\lambda_1+\lambda_{k_0})$ is special, contradicting to ${\bf t}\in T$. Thus, the lemma holds.\end{proof}

\begin{theo}
Let $\sg$ be of type $A_n$ with $n\geq 2$. Then the center $Z(U_q(\sg))$ is isomorphic to the commutative algebra generated by elements $x_i, y_k (1\leq i,k\leq n, k\not=1,r)$ and $w_{\bf t} ({\bf t}\in T)$ with relations \beqs y_k^{r_k}=x_1^{\frac{d_k}{(r,k)}}x_k,\;x_1^{\|{\bf t}\|}w_{\bf t}=\prod_{k\not=1,r}y_{k}^{t_k}.\eeqs 
\end{theo}
\begin{proof}
Let $\ca$ be the $\mc(q)$-algebra generated by elements $x_i, y_k (1\leq i,k\leq n, k\not=1,r)$ and $w_{\bf t} ({\bf t}\in T)$ subject to the relations:
\beqs y_k^{r_k}=x_1^\frac{d_k}{(r,k)}x_k,\;x_1^{\|{\bf t}\|}w_{\bf t}=\prod_{k\not=1,r}y_{k}^{t_k}.\eeqs

Define an algebra homomorphism $\zeta:\ca\rightarrow\mc(q)[\Psi]$ via
\beqs \zeta(x_i)=r_i\lambda_i,&\zeta(y_{k})=d_k\lambda_1+\lambda_k,&\zeta(w_{\bf t})=\xi^{-1}({\bf t})\eeqs
for $1\leq i,k\leq n\,  (k\not=1,r)$ and all non-single non-special minimal ${\bf t}$.

By Lemma 4.5,  we have \beqs\zeta(y_k^{r_k})=r_k(d_k\lambda_1+\lambda_k)=\frac{d_k}{(r,k)}(r\lambda_1)+r_kx_k=\zeta(x_k^\frac{d_k}{(r,k)}x_k).\eeqs
It follows from Lemma 4.6 that
\beqs\zeta(\prod_{k\not=1,r}y_{k}^{t_k})=\sum_{k\not=1,r}t_k{\bf e}(k)=(\sum_{k\not=1,r}t_kd_k\lambda_1)+\sum_{k\not=1,r}t_k\lambda_k=\|{\bf t}\|(r\lambda_1)+\xi^{-1}({\bf t})=\zeta(x_1^{\|{\bf t}\|}w_{\bf t}).\eeqs
Thus,  $\zeta$ preserves the generating relations of $\ca$ and it is well-defined. Moreover,
by the definition of $\Psi_{\min}$, the homomorphism $\zeta$ is an epimorphism.

In order to show that $\zeta$ is  injective, we let $\cb=S^{-1}\ca$ be the localization of $\ca$ at $S=\mc(q)[x_1,\cdots,x_n]\setminus\{0\}$. Then $\zeta$ can be extended to an epimorphism (also denoted by $\zeta$) from $\cb$ to the localization $\zeta(S)^{-1}\mc(q)[\Psi]$ by
\beqs \zeta(s^{-1}a)&=&\zeta(s)^{-1}\zeta(a)\eeqs
for all $s\in S$ and $a\in \ca$. Now it suffices to prove that $\cb$ is a field.

We only prove it for $r=n+1$. The proof for $r=\frac{n+1}{2}$ is  similar.

Let $\mf=\mc(q)(x_1,\cdots,x_n)$ be the fraction field.
Then $\mf_1:=\mf[y_2]$ is a field since $t^{r_2}-x_1^{d_2}x_2$ is irreducible in $\mf[t]$, and $\mf_2:=\mf_1[y_2]=\mf[y_2, y_3]$ is also a field since
$t^{r_3}-x_1^{d_3}x_3$ is irreducible in $\mf_1[t]$. Repeating the same process, we obtain that  $\mf_{n-1}:=\mf[y_2,\cdots, y_{n}]$ is field.

Now for all ordinary ${\bf t}$, we have $w_{\bf t}=x_1^{-\|{\bf t}\|}\prod_{k=2}^ny_{k}^{t_k}\in\mf_{n-1}$. So
$\cb=\mf_{n-1}$ is a field.
\end{proof}

In the end of this paper, we present 3 examples in the cases of  types $A_2, A_3$ and $A_4$ respectively. In these cases, we list all minimal $n$-sequences  of type $r$, a minimal generating set of $\Psi$ corresponding to a minimal generating set of the center $Z(U_q(\sg))$, and the algebraic structure of $Z(U_q(\sg))$.  The center $Z(U_q(\sg))$ for types $A_2$ and $A_3$ have been proved in \cite{LWP} and \cite{WWL} respectively. 
For the type $A_4$,   it turns out that  the center $Z(U_q(\sg))$ is isomorphic to a quotient algebra of polynomial algebra in fourteen variables with ten relations.

\begin{exam}{\rm
Type $A_2$ case.

(1) The minimal generating set of $\Psi$:
$$\Psi_{\min}=\{3\lambda_1, 3\lambda_2, \lambda_1+\lambda_2\}.$$

(2) The minimal $2$-sequences of type $3$:
\beqs single: (3,0), (0,3);\;\quad special: (1,1);\;\quad T=\emptyset.\eeqs

(3)  The center $Z(U_q(\sg)):$ isomorphic to
$$\mc(q)[x_1,x_2,y_2]/(x_1x_2-y_2^3).$$

}
\end{exam}
\begin{exam}
{\rm
Type $A_3$ case .

(1) The minimal generating set of $\Psi$:
$$\Psi_{\min}=\{2\lambda_1, \lambda_2, 2\lambda_3, \lambda_1+\lambda_3\}.$$

(2) The minimal $3$-sequences of type $2$: \beqs single: (2,0,0), (0,1,0), (0,0,2);\quad special: (1,0,1); \quad T=\emptyset.\eeqs

(3) The center $Z(U_q(\sg))$: isomorphic to
$$\mc(q)[x_1,x_2,x_3,y_3]/(x_1x_3-y_3^2).$$

}
\end{exam}
\begin{exam}{\rm
 Type $A_4$ case.

(1) The minimal generating set  of $\Psi$:
\beqs\Psi_{\min}&=&\{5\lambda_1, 5\lambda_2, 5\lambda_3, 5\lambda_4,  3\lambda_1+\lambda_2, 2\lambda_1+\lambda_3, \lambda_1+\lambda_4,\\
 && \lambda_2+\lambda_3, \lambda_1+2\lambda_2, \lambda_1+3\lambda_3, \lambda_2+2\lambda_4, 3\lambda_2+\lambda_4, 2\lambda_3+\lambda_4, \lambda_3+3\lambda_4\}.\eeqs

(2) The minimal $4$-sequences of type $5$:
\beqs &&single: (5,0,0,0), (0,5,0,0), (0,0,5,0), (0,0,0,5);\\
&&special:  (3,1,0,0),  (2,0,1,0), (1,0,0,1);\\
&& T=\{(0,1,1,0), (1,2,0,0), (1,0,3,0), (0,1,0,2), (0,3,0,1),(0,0,2,1), (0,0,1,3)\}.\eeqs

(3) The center $Z(U_q(\sg))$: isomorphic to
$$\mc(q)[x_1,x_2,x_3,x_4,y_2,y_3,y_4,w_1,w_2,w_3,w_4,w_5,w_6,w_7]/I,$$
where $I$ is the ideal generated by
\beqs &&x_1^3x_2-y_2^5,\, x_1^2x_3-y_3^5,\, x_1x_4-y_4^5,\, x_1w_1-y_2y_3,\, x_1w_2-y_2^2,\\
&& x_1w_3-y_{3}^3,\, x_1w_4-y_2y_4^2,\, x_1^2w_5-y_2^3y_4,\, x_1w_{6}-y_3^2y_4,\, x_1w_7-y_3y_4^3.\eeqs

}\end{exam}

\newpage
\section*{Appendix: Fundamental weights and prime roots}

For convenience, we list all the fundamental weights as linear combinations of prime roots in Table 1, and list all the prime roots as linear combinations of fundamental weights in Table 2.
These relations are well known in the representation theory of Lie algebras. For example, one can find the following Table 2  in \cite{Hum}.

$$\begin{tabular}{c}
{\bf Table \quad 1}\\
\begin{tabular}{|c|l|}
\hline Type of $\sg$& \qquad\qquad\qquad Fundamental  weights presented by prime roots\\
\hline $A_n$&$\lambda_i=\frac{n+1-i}{n+1}\Big(\alpha_1+2\alpha_2\cdots+(i-1)\alpha_{i-1}\Big)+\frac{(n+1-i)i}{n+1}\alpha_i$\\
&\qquad$+\frac{i}{n+1}\Big(\alpha_n+2\alpha_{n-1}+\cdots+(n-i)\alpha_{i+1}\Big)$, $1\leq i\leq n$\\
\hline $B_n$&$\lambda_i=\alpha_1+2\alpha_2+\cdots+i(\alpha_i+\alpha_{i+1}+\cdots+\alpha_n)$, $1\leq i\leq n-1$\\
&$\lambda_n=\frac{1}2(\alpha_1+2\alpha_2+\cdots+i\alpha_i+\cdots+n\alpha_n)$\\
\hline $C_n$&$\lambda_i=\alpha_1+2\alpha_2+\cdots+i(\alpha_i+\alpha_{i+1}+\cdots+\alpha_{n-1}+\frac12\alpha_n)$, $1\leq i\leq n-1$\\ &$\lambda_n=\alpha_1+2\alpha_2+\cdots+(n-1)\alpha_{n-1}+\frac{n}2\alpha_n$\\
\hline $D_n$&$\lambda_i=\alpha_1+2\alpha_2+\cdots+i(\alpha_i+\alpha_{i+1}+\cdots+\alpha_{n-2})+\frac{i}2(\alpha_{n-1}+\alpha_n)$, $1\leq i\leq n-2$\\
&$\lambda_{n-1}=\frac12\Big(\alpha_1+2\alpha_2+\cdots+(n-2)\alpha_{n-2}+\frac{n}2\alpha_{n-1}+\frac{n-2}2\alpha_n\Big)$\\
&$\lambda_n=\frac12\Big(\alpha_1+2\alpha_2+\cdots+(n-2)\alpha_{n-2}+\frac{n-2}2\alpha_{n-1}+\frac{n}2\alpha_n\Big)$\\
\hline $E_6$&$\lambda_1=\frac13(4\alpha_1+3\alpha_2+5\alpha_3+6\alpha_4+4\alpha_5+2\alpha_6)$,\\
& $\lambda_2=\alpha_1+2\alpha_2+2\alpha_3+3\alpha_4+2\alpha_5+\alpha_6$\\
&$\lambda_3=\frac13(5\alpha_1+6\alpha_2+10\alpha_3+12\alpha_4+8\alpha_5+4\alpha_6)$,\\
& $\lambda_4=2\alpha_1+3\alpha_2+4\alpha_3+6\alpha_4+4\alpha_5+2\alpha_6$\\
&$\lambda_5=\frac13(4\alpha_1+6\alpha_2+8\alpha_3+12\alpha_4+10\alpha_5+5\alpha_6)$, \\
&$\lambda_6=\frac13(2\alpha_1+3\alpha_2+4\alpha_3+6\alpha_4+5\alpha_5+4\alpha_6)$\\
\hline $E_7$&$\lambda_1=2\alpha_1+2\alpha_2+3\alpha_3+4\alpha_4+3\alpha_5+2\alpha_6+\alpha_7$\\
& $\lambda_2=\frac12(4\alpha_1+7\alpha_2+8\alpha_3+12\alpha_4+9\alpha_5+6\alpha_6+3\alpha_7)$\\
&$\lambda_3=3\alpha_1+4\alpha_2+6\alpha_3+8\alpha_4+6\alpha_5+4\alpha_6+2\alpha_7$,\\
& $\lambda_4=4\alpha_1+6\alpha_2+8\alpha_3+12\alpha_4+9\alpha_5+6\alpha_6+3\alpha_7$\\
&$\lambda_5=\frac12(6\alpha_1+9\alpha_2+12\alpha_3+18\alpha_4+15\alpha_5+10\alpha_6+5\alpha_7)$,\\
& $\lambda_6=2\alpha_1+3\alpha_2+4\alpha_3+6\alpha_4+5\alpha_5+4\alpha_6+2\alpha_7$\\
& $\lambda_7=\frac12(2\alpha_1+3\alpha_2+4\alpha_3+6\alpha_4+5\alpha_5+4\alpha_6+3\alpha_7)$\\
\hline $E_8$&$\lambda_1=4\alpha_1+5\alpha_2+7\alpha_3+10\alpha_4+8\alpha_5+6\alpha_6+4\alpha_7+2\alpha_8$\\
&$\lambda_2=5\alpha_1+8\alpha_2+10\alpha_3+15\alpha_4+12\alpha_5+9\alpha_6+6\alpha_7+3\alpha_8$\\
&$\lambda_3=7\alpha_1+10\alpha_2+14\alpha_3+20\alpha_4+16\alpha_5+12\alpha_6+8\alpha_7+4\alpha_8$\\
&$\lambda_4=10\alpha_1+15\alpha_2+20\alpha_3+30\alpha_4+24\alpha_5+18\alpha_6+12\alpha_7+6\alpha_8$\\
&$\lambda_5=8\alpha_1+12\alpha_2+16\alpha_3+24\alpha_4+20\alpha_5+15\alpha_6+10\alpha_7+5\alpha_8$\\
&$\lambda_6=6\alpha_1+9\alpha_2+12\alpha_3+18\alpha_4+15\alpha_5+12\alpha_6+8\alpha_7+4\alpha_8$\\
&$\lambda_7=4\alpha_1+6\alpha_2+8\alpha_3+12\alpha_4+10\alpha_5+8\alpha_6+6\alpha_7+3\alpha_8$\\
&$\lambda_8=2\alpha_1+3\alpha_2+4\alpha_3+6\alpha_4+5\alpha_5+4\alpha_6+3\alpha_7+2\alpha_8$\\
\hline $F_4$&$\lambda_1=2\alpha_1+3\alpha_2+4\alpha_3+2\alpha_4$, $\lambda_2=3\alpha_1+6\alpha_2+8\alpha_3+4\alpha_4$\\
&$\lambda_3=2\alpha_1+4\alpha_2+6\alpha_3+3\alpha_4$, $\lambda_4=\alpha_1+2\alpha_2+3\alpha_3+2\alpha_4$\\
\hline $G_2$&$\lambda_1=2\alpha_1+3\alpha_2$, $\lambda_2=\alpha_1+2\alpha_2$\\
\hline\end{tabular}\end{tabular}$$

$${\small\begin{tabular}{c}
{\bf Table \quad 2}\\
\begin{tabular}{|c|l|}
\hline Type of $\sg$& \qquad\qquad\qquad Prime roots presented by fundamental  weights\\
\hline $A_1$&$\alpha_1=2\lambda_1$\\
\hline $A_n(n>1)$&$\alpha_1=2\lambda_1-\lambda_2, \alpha_n=-\lambda_{n-1}+2\lambda_n$,\\
&$ \alpha_i=-\lambda_{i-1}+2\lambda_i-\lambda_{i+1}, 2\leq i\leq n-1$\\
\hline $B_2$&$\alpha_1=2\lambda_1-2\lambda_2, \alpha_2=-\lambda_1+2\lambda_2$\\
\hline $B_n(n>2)$&$\alpha_1=2\lambda_1-\lambda_2, \alpha_{n-1}=-\lambda_{n-2}+2\lambda_{n-1}-2\lambda_n, \alpha_n=-\lambda_{n-1}+2\lambda_n$,\\
& $\alpha_i=-\lambda_{i-1}+2\lambda_i-\lambda_{i+1}, 2\leq i\leq n-2$\\
\hline $C_n$&$\alpha_1=2\lambda_1-\lambda_2, \alpha_n=-2\lambda_{n-1}+2\lambda_n$,\\
& $\alpha_i=-\lambda_{i-1}+2\lambda_i-\lambda_{i+1}, 2\leq i\leq n-1$\\
\hline $D_n$&$\alpha_1=2\lambda_1-\lambda_2, \alpha_i=-\lambda_{i-1}+2\lambda_i-\lambda_{i+1}$, $2\leq i\leq n-3$\\
&$\alpha_{n-2}=-\lambda_{n-3}+2\lambda_{n-2}-\lambda_{n-1}-\lambda_n$, \\
&$\alpha_{n-1}=-\lambda_{n-2}+2\lambda_{n-1}, \alpha_{n}=-\lambda_{n-2}+2\lambda_{n}$\\
\hline $E_6$&$\alpha_1=2\lambda_1-\lambda_3, \alpha_2=2\lambda_2-\lambda_4, \alpha_3=-\lambda_1+2\lambda_3-\lambda_4$,\\
 &$\alpha_4=-\lambda_2-\lambda_3+2\lambda_4-\lambda_5, \alpha_5=-\lambda_4+2\lambda_5-\lambda_6, \alpha_6=-\lambda_5+2\lambda_6$ \\
\hline $E_7$&$\alpha_1=2\lambda_1-\lambda_3, \alpha_2=2\lambda_2-\lambda_4, \alpha_3=-\lambda_1+2\lambda_3-\lambda_4$,\\
 &$\alpha_4=-\lambda_2-\lambda_3+2\lambda_4-\lambda_5, \alpha_5=-\lambda_4+2\lambda_5-\lambda_6$,\\
 &$\alpha_6=-\lambda_5+2\lambda_6-\lambda_7$, $\alpha_7=-\lambda_6+2\lambda_7$ \\
\hline $E_8$&$\alpha_1=2\lambda_1-\lambda_3, \alpha_2=2\lambda_2-\lambda_4, \alpha_3=-\lambda_1+2\lambda_3-\lambda_4$,\\
 &$\alpha_4=-\lambda_2-\lambda_3+2\lambda_4-\lambda_5, \alpha_5=-\lambda_4+2\lambda_5-\lambda_6$,\\
 &$\alpha_6=-\lambda_5+2\lambda_6-\lambda_7$, $\alpha_7=-\lambda_6+2\lambda_7-\lambda_8$, $\alpha_8=-\lambda_7+2\lambda_8$ \\
\hline $F_4$&$\alpha_1=2\lambda_1-\lambda_2, \alpha_2=-\lambda_1+2\lambda_2-2\lambda_3$, $\alpha_3=-\lambda_2+2\lambda_3-\lambda_4, \alpha_4=-\lambda_3+2\lambda_4$\\
\hline $G_2$&$\alpha_1=2\lambda_1-3\lambda_2$, $\alpha_2=-\lambda_1+2\lambda_2$\\
\hline\end{tabular}\end{tabular}}$$

\vskip30pt
\def\refname{

\end{document}